\documentclass[11pt]{amsart}
\usepackage{amsmath, amsthm, amssymb,amscd}
\usepackage{float}
\usepackage{graphicx}

\usepackage[all,cmtip]{xy}
\usepackage[english]{babel}

\newcommand{\seq}{\subseteq}

\newtheorem{thm}{Theorem}[section]
\newtheorem*{thm-nl}{Theorem}
\newtheorem*{prop-nl}{Proposition}
\newtheorem{definition}[thm]{Definition}
\newtheorem{lem}[thm]{Lemma}

\def\PP{{\textbf P}}
\def\OO{\mathcal{O}}

\def\T{\mathcal{T}}

\def\cR{\mathcal{R}}

\def\Pic0{{\rm Pic}^0(X)}

\def\mm{\overline{\mathcal{M}}}

\newtheorem*{cor-nl}{Corollary}

\newtheorem*{conjecture-nl}{Conjecture}
\newtheorem*{quest-nl}{Question}
\newtheorem*{quests-nl}{Questions}

\newtheorem{prop}[thm]{Proposition}

\theoremstyle{remark}

\newtheorem{remark}[thm]{Remark}

\title[The Prym--Green Conjecture for torsion line bundles of high order]{The Prym--Green Conjecture for torsion line bundles of high order}

\author[G. Farkas]{Gavril Farkas}
\address{Humboldt-Universit\"at zu Berlin, Institut f\"ur Mathematik,  Unter den Linden 6
\hfill \newline\texttt{}
 \indent 10099 Berlin, Germany} \email{{\tt farkas@math.hu-berlin.de}}

\author[M. Kemeny]{Michael Kemeny}

\address{Humboldt-Universit\"at zu Berlin, Institut f\"ur Mathematik,  Unter den Linden 6
\hfill \newline\texttt{}
 \indent 10099 Berlin, Germany} \email{{\tt michael.kemeny@gmail.com}}

\begin{document}
\bibliographystyle{plain}

\begin{abstract}
Using a construction of Barth and Verra that realizes torsion bundles on sections of special $K3$ surfaces, we prove that the minimal resolution of a general paracanonical curve $C$ of odd genus $g$ and order $\ell\geq \sqrt{\frac{g+2}{2}}$ is natural, thus proving the Prym--Green Conjecture of \cite{CEFS} . In the process, we confirm  the expectation  of Barth--Verra concerning the number of curves with $\ell$-torsion line bundle in a linear system on a special $K3$ surface.
\end{abstract}

\maketitle
A \emph{paracanonical} curve of genus $g$ and order $\ell\geq 2$ is a smooth curve $C$
$$\phi_{K_C\otimes \eta}:C\hookrightarrow \PP^{g-2}$$
embedded by the paracanonical linear system $K_C\otimes \eta$, where $\eta \in \mbox{Pic}^0(C)$ is a torsion line bundle of order  $\ell$. Pairs $[C, \eta]$, which we often refer to as \emph{smooth level $\ell$ curves}, form an irreducible  moduli space $\cR_{g, \ell}$, whose birational geometry has been studied in \cite{CEFS} and \cite{CF}. For integers $p,q\geq 0$, we introduce the Koszul cohomology group of $p$-th syzygies of weight $q$
$$K_{p,q}(C, K_C\otimes \eta)=\mbox{Tor}^p_S\bigl(\Gamma_C(K_C\otimes \eta), \mathbb C\bigr)_{p+q},$$
where $S:=\mbox{Sym } H^0(C,K_C\otimes \eta)=\mathbb C[x_1, \ldots, x_{g-1}]$ is the polynomial algebra and
$$\Gamma_C(K_C\otimes \eta):=\bigoplus_{n\geq 0} H^0\Bigl(C,(K_C\otimes \eta)^{\otimes n}\Bigr)$$
is the homogeneous paracanonical coordinate ring, viewed as a graded $S$-module. The graded Betti numbers of the paracanonical curve $[C, \eta]$ are defined by $b_{p,q}:=\mbox{dim } K_{p,q}(C, K_C\otimes \eta)$.

\vskip 4pt

The main result of this paper is a proof of the Prym--Green Conjecture formulated in \cite{CEFS} for general paracanonical curves of odd genus $g$ of all but finitely many levels $\ell$.

\begin{thm}\label{pg}
Let $[C, \eta] \in \mathcal{R}_{g,\ell}$ be a general level $\ell$ curve of odd genus $g$, such that $\ell \geq \sqrt{\frac{g+2}{2}}$. Then the resolution of the associated paracanonical curve $C\subset \PP^{g-2}$ is natural, that is, the following vanishings of Koszul cohomology groups hold:
$$ K_{\frac{g-3}{2},1}(C, K_C \otimes \eta)=0 \; \; \text{and} \; \; K_{\frac{g-7}{2},2}(C, K_C \otimes \eta)=0$$
\end{thm}

The naturality of the resolution amounts to the vanishing $b_{p,2} \cdot b_{p+1,1}=0$ for all $p$, that is, the product of Betti numbers on each diagonal in the Betti diagram of the paracanonical curve vanishes. The Prym--Green Conjecture in odd genus and level $\ell=2$ (that is, for classical Prym-canonical curves) has been proved in \cite{FK}. Theorem \ref{pg} completely determines the minimal free resolution of the paracanonical ideal $I_{C}\subset S$ in genus $g=2i+5$, which has the following shape:


\begin{table}[htp!]
\begin{center}
\begin{tabular}{|c|c|c|c|c|c|c|c|c|}
\hline
$1$ & $2$ & $\ldots$ & $i-1$ & $i$ & $i+1$ & $i+2$  & $\ldots$ & $2i+2$\\
\hline
$b_{1,1}$ & $b_{2,1}$ & $\ldots$ & $b_{i-1,1}$ & $b_{i,1}$ & 0 & 0 &  $\ldots$ & 0 \\
\hline
$0$ &  $0$ & $\ldots$ & $0$ & $b_{i,2}$ & $b_{i+1,2}$ & $b_{i+2,2}$ & $\ldots$ & $b_{2i+2,2}$\\
\hline
\end{tabular}
\end{center}
    \label{tab:even}
\end{table}

where,
 $$b_{p,1}=\frac{p(2i-2p+1)}{2i+3}{2i+4\choose p+1} \ \  \mbox{ for } p\leq i,\  \mbox{ and } \ \  b_{p,2}=\frac{(p+1)(2p-2i+1)}{2i+3}{2i+4\choose p+2} \  \ \  \mbox{ for } p\geq i.$$

\vskip 5pt

Note that the resolution is \emph{natural}, but fails to be \emph{pure} in column $i$, for both Koszul cohomology groups $K_{i,1}(C,K_C\otimes \eta)$ and $K_{i,2}(C,K_C\otimes \eta)$ are non-zero. In this sense, the resolution of the general level $\ell$ paracanonical curve of \emph{odd} genus has the same shape as the resolution of the general canonical curve of \emph{even} genus, which verifies Green's Conjecture \cite{V1}.

\vskip 6pt

The method of proving Theorem \ref{pg}, can be employed to give a partial solution to the Prym--Green Conjecture on paracanonical curves of even genus as well. We write $g=2i+6$, and since it can easily be computed that $\mbox{dim } K_{i,2}(C,K_C\otimes \eta)=\mbox{dim } K_{i+1,1} (C,K_C\otimes \eta)$ for any $[C, \eta]\in \cR_{g,\ell}$, the Prym-Green conjecture from \cite{CEFS} amounts to
a single vanishing statement, namely
\begin{equation}K_{i,2}(C,K_C\otimes \eta)=0.
\end{equation}
Equivalently, the paracanonical curve $C\subset \PP^{g-2}$ verifies the Green-Lazarsfeld property $(N_i)$.

\vskip 4pt

The Prym--Green Conjecture is expected to fail for $\ell=2$ and $g=2^n\geq 8$. The reasons for this failure are not yet understood geometrically. In \cite{CEFS} the conjecture has been checked using Macaulay2 for levels $\ell\leq 5$ and $g\leq 18$, with the two already mentioned exceptions $\ell=2$ and $g\in \{8, 16\}$. One expects no further exceptions for levels $\ell\geq 3$, that is, the resolution of the general paracanonical curve should always be natural.
Here we prove a weaker version of the conjecture in even genus, subject to the same numerical restrictions concerning the level as those in Theorem \ref{pg}.

\begin{thm}\label{pgeven}
Let $g \geq 8$ be even and $\ell$ an integer level such that $\ell=2$ or $\ell\geq \sqrt{\frac{g+2}{2}}$. Then for a general level $\ell$ curve $[C, \eta]\in \cR_{g,\ell}$ the following vanishings hold:
$$K_{\frac{g}{2}-4,2}(C,K_C\otimes \eta)=0 \ \ \mbox{ and } \ \ K_{\frac{g}{2}-1,1}(C,K_C\otimes \eta)=0.$$
\end{thm}

\vskip 4pt

The conclusion of Theorem \ref{pgeven} can be reformulated in terms of the graded Betti table of a general level $\ell$ paracanonical curve $\phi_{K_C\otimes \eta}:C\hookrightarrow \PP^{2i+4}$, which has the following shape:

\begin{table}[htp!]
\begin{center}
\begin{tabular}{|c|c|c|c|c|c|c|c|c|c|}
\hline
$1$ & $2$ & $\ldots$ & $i-1$ & $i$ & $i+1$ & $i+2$ & $i+3$ & $\ldots$ & $2i+3$\\
\hline
$b_{1,1}$ & $b_{2,1}$ & $\ldots$ & $b_{i-1,1}$ & $b_{i,1}$ & ? & 0 & 0 &  $\ldots$ & 0 \\
\hline
$0$ &  $0$ & $\ldots$ & $0$ & ? & $b_{i+1,2}$ & $b_{i+2,2}$ & $b_{i+3,2}$ & $\ldots$ & $b_{2i+3,2}$\\
\hline
\end{tabular}
\end{center}
    \caption{The Betti table of a general paracanonical curve of genus $g=2i+6$}
    \label{tab:even}
\end{table}
where,
$$b_{p,1}={2i+5\choose p+1}\frac{p(i-p+1)}{i+2} \ \mbox{ for } \ p\leq i,\  \mbox{ and } \ b_{p,2}={2i+5\choose p+2}\frac{(p+1)(p-i)}{i+2}  \  \mbox{ for } \ p\geq i+1.$$
The question marks in Table \ref{tab:even} refer to the common value $b_{i+1,1}=b_{i,2}$, which is expected to be zero, at least when $\ell\geq 3$. To settle this last
case of the Prym-Green Conjecture and, in particular, explain the puzzling exceptions in level $2$, genuinely new ideas seem to be required.

\vskip 6pt

The proof of Theorems \ref{pg} and \ref{pgeven} in the case when $\ell$ is relatively high with respect to $g$ uses in an essential way special $K3$ surfaces and a beautiful idea of Barth--Verra \cite{BV}. Contrary to the approach via Nikulin surfaces employed in \cite{FK} when $\ell=2$, in this paper we take the opposite approach of fixing the genus and instead letting the level grow high with respect to the genus.

\vskip 6pt

In what follows, we describe the Barth--Verra construction. We fix a smooth $K3$ surface $X_g$ such that  $\mbox{Pic}(X_g)\supseteq \mathbb Z\cdot L\oplus \mathbb Z\cdot H$, where
$$L^2=2g-2, \ \ L\cdot H=2g-2, \ \ H^2=2g-6,$$
then set $\eta:=H-L\in \mbox{Pic}(X_g)$.  For each curve $C\in |L|$, let $\eta_C:=\eta\otimes \OO_C\in \mbox{Pic}^0(C)$ be the restriction. Since $\mbox{dim } |L|=g$ and $\mbox{Pic}^0(C)$ has also dimension $g$, one expects the set
$$\mathcal{T}_{\ell}:=\Bigl\{C\in |L|: \eta_C^{\otimes \ell}\cong \OO_C\Bigr\}$$
to be finite for each $\ell\geq 2$. Using a Chern class computation, it is shown in \cite{BV} that $\T_{\ell}$ is indeed finite and
\begin{equation}\label{expect1}
|\mathcal{T}_{\ell}|={2\ell^2-2 \choose g},
\end{equation}
provided the following two assumptions are satisfied:

\vskip 4pt

(i) all curves $C\in |L|$ are irreducible, \ and

(ii) each curve $C\in \T_{\ell}$ contributes with multiplicity one.

\vskip 4pt

Condition (ii) is to be understood with respect to the scheme structure on $\mathcal{T}_{\ell}$ as a subvariety of a certain Segre variety defined in \S 3 of \cite{BV}. Whereas assumption (i) is clearly satisfied when for instance $\mbox{Pic}(X_g)=\mathbb Z\cdot L\oplus \mathbb Z\cdot H$, assumption (ii) seems more delicate and a priori one knows nothing about the singularities of the curves $C\in \T_{\ell}$. On the one hand, the  formula (\ref{expect1}) counts the number of divisors $C\in |L|$ such that $\eta_{C} $ has order \emph{dividing} $\ell$, since for composite  $\ell$ torsion points of orders dividing $\ell$ also  contribute. On the other hand, already for $\ell=2$, when $X_g$ is a polarized Nikulin surface of genus $g$ in the sense of \cite{vGS} or \cite{FV}, clearly $\T_{2}=|L|$ is $g$-dimensional, although the expected number computed by formula (\ref{expect1}) equals zero, as soon as $g\geq 7$. One of the main results of this paper is to show that the conjecture (ii) of Barth--Verra does indeed hold, when $X_g$ is general in moduli, and all curves $C\in \T_{\ell}$ are smooth. Rather than following the suggestion from M. Green's letter in \S 8 of \cite{BV} and use Griffiths' infinitesimal invariant for normal functions,  we instead deal with the transversality issues by explicit degenerations to elliptic $K3$ surfaces. We explain our results.

\vskip 4pt
\begin{definition}\label{bv2}
We fix $g\geq 3$ and let $\Upsilon_g$ be the rank two lattice generated by elements $L, \eta$ with $$L^2=2g-2, \; \eta^2=-4, \; L \cdot \eta=0.$$  A smooth $K3$ surface $X_g$ with $\mathrm{Pic}(X_g) \cong \Upsilon_g$ and with  $L$ big and nef is said to be a polarized Barth--Verra  surface of genus $g$.
\end{definition}

Observe that in Definition \ref{bv2}, the condition that $L$ be big and nef only amounts to choosing a sign for $L$ and does not impose additional conditions on the lattice.
Key in our proof of conjecture (ii) from \cite{BV} is a specialization of $X_g$ to an elliptic surface. We consider the lattice $\Omega_g$ with ordered basis $\{ E, \Gamma, \eta \}$ and having intersection numbers
$$E^2=0, \ \ \Gamma^2=-2, \ \  \eta^2=-4, \ \ E\cdot \Gamma=1, \ \ E\cdot \eta=\Gamma\cdot \eta=0.$$
Let $Y_g$ denote a general $\Omega_g$-polarized K3 surface; as we have a primitive embedding
$$\Upsilon_g \hookrightarrow \Omega_g, \ \  \ L \mapsto \Gamma+gE, \  \ \eta \mapsto \eta,$$ we may deform $X_g$ to $Y_g$.
By studying the Mordell--Weil group $MW(Y_g)$ of the elliptic fibration $Y_g \to \PP^1$ induced by the pencil $|E|$, we show:
\begin{thm}\label{genus-one-case}
For each $\ell \geq 2$, there exist finitely many elliptic fibres $F \in |E|$, such that $\eta_F\in \mathrm{Pic}^0(F)[\ell]-\{\mathcal{O}_F\}$.
\end{thm}

Every divisor in the linear system $|\Gamma+gE|$ on the surface $Y_g$ decomposes as a sum $$\Gamma+E_1+\cdots+E_g,$$ for $E_i \in |E|$, where we allow repetitions. This enables us to use Theorem \ref{genus-one-case} in order to describe the space of  divisors $C \in |\Gamma+gE|$ on $Y_g$, such that $\eta_{C}$ is an $\ell$-torsion line bundle. Under a certain transversality assumption which is verified in Section \ref{transversality},  all these divisors end up being reduced, nodal with all nodes disconnecting, see Proposition \ref{nonreduced-prop}. Furthermore there are exactly ${2\ell^2-2 \choose g }$ divisors $C \in |\Gamma+gE|$ such that $\eta_{C}^{\otimes \ell}\cong \OO_C$, which correspond to choosing $g$ of the $2\ell^2-2$ points of intersections $\Gamma\cdot T_{\ell}$, where $T_{\ell}$ is the section  of $Y_g\rightarrow \PP^1$ corresponding to  $\ell$ times a generator of $MW(Y_g) \cong  \mathbb Z$ and $\Gamma \in MW(Y_g)$ is the neutral element. Via a further specialization to Kummer surfaces, we establish in Section \ref{transversality}, the transversality assumption necessary to complete the proof of the conjecture in \cite{BV}.

\vskip 3pt

Before stating the next result, we recall that $\mu(n)$ denotes the M\"obius function. We also use the convention $\displaystyle{a \choose b}=0$ if $b>a$.
\begin{thm} \label{generic-transversal}
For a general Barth-Verra  surface $X_g$ of genus $g \geq 3$,  there exist precisely $\displaystyle{ 2\ell^2-2 \choose g }$  curves $C \in |L|$ such that $\eta_{C}^{\otimes \ell}\cong \OO_C$. All such curves $C$ are smooth and irreducible. The number of curves $C$ such that $\eta_{C}$ has order exactly $\ell$ is strictly positive and given by the following formula:
\begin{equation*}
\sum_{d|\ell} \mu\Bigl(\frac{\ell}{d}\Bigr) {2d^2-2 \choose g}.
\end{equation*}
\end{thm}

\vskip 4pt

We  apply Theorem \ref{generic-transversal} in order to construct a point in the moduli space $\cR_{g, \ell}$, for which the Prym--Green conjecture from \cite{CEFS} can be shown to hold. We take a general Barth--Verra  surface $X_g$ with $\mbox{Pic}(X_g)=\mathbb Z\cdot L\oplus \mathbb Z\cdot \eta$ as above, then choose one of the (smooth) curves $C\in |L|$ such that $\eta_C$ has order $\ell$, thus $[C, \eta_C]\in \cR_{g, \ell}$.

\vskip 3pt

Following broadly the lines of \cite{FK}, we then check that for odd $g$,  the following  Koszul cohomology groups vanish:
$$ K_{\frac{g-3}{2},1}(C, K_C \otimes \eta)=0 \; \; \text{and} \; \; K_{\frac{g-7}{2},2}(C, K_C \otimes \eta)=0.$$

The case of curves of even genus $g$ is similar and we show in Section 4, that for a section $[C, \eta_C]\in \cR_{g, \ell}$  of a Barth--Verra surface as above, the following statements hold:
$$K_{\frac{g}{2}-4,2}(C,K_C\otimes \eta)=0 \ \ \mbox{ and } \ K_{\frac{g}{2}-1,1}(C,K_C\otimes \eta_C)=0.$$

\vskip 4pt

\noindent {\bf{Acknowledgements:}} We are grateful to  Remke Kloosterman for several helpful comments on a draft version of this paper. In particular, the simple proof of Lemma \ref{embedding-is-prim} is due to him. We thank Alessandro Verra for drawing, in a different context, our attention  to the beautiful paper \cite{BV}. We thank the two referees for several helpful corrections and improvements. In particular, a strengthening of the original version of Theorem \ref{generic-transversal} is due to a remark of the referee. This work was supported by the DFG Priority Program 1489 \emph{Algorithmische Methoden in Algebra, Geometrie und Zahlentheorie}.

\section{Moduli spaces of sheaves on K3 surfaces}
In this section we gather some results on moduli spaces of sheaves on $K3$ surfaces that we shall need and refer to  \cite{HL} for background. We shall use these moduli  spaces to study families of line bundles on (possibly reducible and/or non-reduced) divisors on $K3$ surfaces. If $X$ is a variety, $C\subseteq X$ and $A\in \mbox{Pic}(X)$ is a line bundle, we set $A_C:=A\otimes \OO_C$.

\begin{lem} \label{stability}
Let $(X,H)$ be a smooth polarized $K3$ surface and  $i:C \hookrightarrow X$ a curve in the linear system $|H|$. Let $A \in \mathrm{Pic}(X)$ be a line bundle. Assume that for any effective subcurve $C' \subsetneq C$, we have
\begin{enumerate}
\item $A \cdot C'=A \cdot C$,
\item if $C'$ nontrivial with $C' \neq C$, then $(C' \cdot C)>(C')^2$.
\end{enumerate} Then $i_*(A_C)$ is a stable coherent sheaf on $X$ with Mukai vector $v=\Bigl(0, H, (A \cdot H)-(H)^2/2\Bigr)$.
\end{lem}
\begin{proof}
We set $N:=A_C$ and first show that $i_*N$ is \emph{pure}, i.e.\ the support of any non-trivial coherent subsheaf $B \seq i_*N$ is of pure dimension $1$. It is clearly enough to show that $\mathcal{O}_C$ is pure and for this it suffices to show that is has no embedded points, see \cite[Pg.\ 3]{HL}. Since $C$ is a divisor in a smooth surface, it is Cohen-Macaulay, so that $C$ contains no embedded points, see for instance  \cite{stacks}, Lemma 30.4.4.


\vskip 4pt

We now establish the stability of $i_*(N)$. Let $B \seq i_*N$ be a coherent proper sheaf. By adjunction, we may write $B=i_*(B')$ for a coherent subsheaf $B' \seq N$. We have an exact sequence
$$0 \longrightarrow B' \otimes N^{\vee} \longrightarrow  \mathcal{O}_C \longrightarrow  \mathcal{O}_{Z} \longrightarrow  0, $$
where $Z \seq C$ is a closed subscheme. Let $C'$ be the complement of the union of all zero-dimensional components of $Z$; this is a subdivisor of $C$ and $C' \neq C$ as $B'$ has one-dimensional support. We have
$c_1(i_* N_{Z})=c_1(i_* N_{{C'}})$, since a divisor is determined by its restriction to an open set whose complement has codimension two. From the exact sequence
$$0 \longrightarrow  A(-C') \longrightarrow  A \longrightarrow  i_* N_{{C'}} \longrightarrow  0, $$
we obtain $c_1(i_* N_{Z})=C'$ and likewise $c_1(i_*N)=C$. Thus $c_1(B)=C-C'$. From Hirzebruch--Riemann--Roch, for any coherent sheaf $E$ of rank $r(E)$ on a $K3$ surface, one can express its Hilbert polynomial with respect to $H$ as
$$ P(E,m)=\frac{r(E)^2}{2}m^2+m(H \cdot c_1(E))+\chi(E),$$
since $c_1(X)=0$. Thus in the case $r(E)=0$, the reduced Hilbert polynomial, \cite[Def.\ 1.2.3]{HL},  has the form
$p(E,m)=m+\frac{\chi(E)}{(H \cdot c_1(E))}$. Hence, to prove stability of $i_*N$, we need to show that
$$ \frac{\chi(B')}{(C-C') \cdot C} < \frac{\chi(N)}{C^2}. $$ If $C'$ is trivial, then $Z$ is non-empty and has zero-dimensional support, so
$$\chi(B')=\chi(N)-\chi(N_{Z})<\chi(N)$$ and the claim holds. Thus we may assume $0\neq C'\neq C$. We have $\chi(B')=\chi(N)-\chi(N_{Z}) \leq \chi(N)-\chi(N_{{C'}})$. Thus it suffices to show $$\chi(N_{{C'}})\cdot (C)^2 > \chi(N)\cdot (C' \cdot C).$$ We have
$2\chi(N_{{C'}})=2\chi(A)-2\chi(A(-C'))=2(A \cdot C')-(C')^2$ and likewise, we write that
$2\chi(N)=2(A \cdot C)-(C)^2.$ The claim now follows from the assumption $(A \cdot C')=(A \cdot C)$, as well as the inequalities $(C)^2 > (C' \cdot C)$, which follows from the ampleness of
the divisor $C \in |H|$ and the inequality $(C' \cdot C) > (C')^2$ respectively, which holds by assumption.
\end{proof}

\vskip 3pt

Let $X, A$ and $C$ as above and set $v:=\Bigl(0, H, (A \cdot H)-(H)^2/2\Bigr)$. The coarse moduli space $M_X(v)$ of semistable sheaves on $X$ with Mukai vector $v$ is known to be projective of dimension $v^2+2=(H)^2+2=:2g$. The open subset $M_X^s(v)$ consisting of stable sheaves is known to be smooth. For a curve $C\subset X$, we denote by $i:C\hookrightarrow X$ the inclusion map.

\begin{lem} \label{locus-via-sheaves}
Let $(X,H)$ be a smooth polarized  $K3$ surface and a line bundle $A \in \mathrm {Pic}(X)$ with $(A \cdot H)=0$. Assume that every curve $C \in |H|$ satisfies the assumptions of Lemma \ref{stability} with respect to $A$. Then the following closed set
$$Z:=\Bigl\{i_*(A_{C}) \cong i_* (\OO_C): C\in |H|\Bigr\} \seq M_X^s(v)$$ is the intersection of two closed subsets $Z_1$ and $Z_2$ of $M_X^s(v)$ both of dimension $g$. In particular, $Z$ has expected dimension zero.
\end{lem}
\begin{proof}
Under our assumptions both sheaves $i_*(A_{C})$ and $i_*(\mathcal{O}_{C})$ are stable with the same Mukai vector $v$. Let $Z_1$ be the closed locus of $M_X(v)$ parametrizing points of the form $i_*(A_{C})$ and $Z_2$ the locus parametrizing points of the form $i_*(\OO_C)$. Both $Z_1$ and $Z_2$  are of codimension $g$ inside $M_X(v)$. Then $Z=Z(X,H,A):=Z_1\cap Z_2$.
\end{proof}

\begin{remark}\label{relative1}
Lemma \ref{locus-via-sheaves} will be used later in a relative setting. Precisely, if $(X_t, H_t, A_t)_{t\in T}$ is a flat family of objects as in Lemma \ref{locus-via-sheaves} over an irreducible base $T$ for which  there exists a point $0\in T$ such that the subvariety $Z(X_0,H_0,A_0)$ is $0$-dimensional, then since the construction of the varieties $Z_1$ and $Z_2$ can clearly be done in relative setting, it follows that $Z(X_t, H_t, A_t)$ is $0$-dimensional for a general $t\in T$.
\end{remark}

\section{The elliptic $K3$ surface $Y_g$}
In this section we study the elliptic $\Omega_g$-polarized K3 surface from the introduction. By definition, $Y_g$ is a general $K3$ surface with Picard group given by the rank three lattice $\Omega_g$ having an ordered basis $\{ E, \Gamma, \eta \}$ and with the following corresponding intersection matrix
$$
\left( \begin{array}{ccc}
0 & 1 & 0 \\
1 & -2 & 0 \\
0 & 0 & -4 \end{array} \right),$$
such that $\Gamma+gE$ is a big and nef class. Such $K3$ surfaces exist by the Torelli theorem and by results of Nikulin, see for instance \cite{dolgachev} or \cite{Mo}.

\begin{lem} \label{elliptic-1}
If $Y_g$ is a general $\Omega_g$-polarized $K3$ surface, then one may choose the basis $\{ E, \Gamma, \eta \}$ of $\mathrm{Pic}(Y_g)$ such that $E$, $\Gamma$ are represented by integral curves. Furthermore all fibres of $|E|$ are integral and nodal and $\Gamma$ avoids the nodes.
\end{lem}
\begin{proof}
Consider the elliptic K3 surface $S \to \PP^1$ described by the Weierstrass equation
$$ y^2=x(x^2+a(t)x+b(t)),$$
for general polynomials $a(t)$ respectively  $b(t)$, of degree $4$ respectively $8$; this has already been studied in \cite[\S 4]{vGS} and  \cite[\S 3]{HK}. The surface $S$ has two disjoint sections $\sigma$ respectively \ $\tau$ given by $x=z=0$ respectively \ $x=y=0$. Further $S$ has $16$ singular fibres, all nodal, $8$ of which are integral and $8$ of which have two components. Let $N_1, \ldots, N_8$ be the components of the reducible singular fibres avoiding $\sigma$, and let $F$ denote the class of a general fibre of $S \to \PP^1$. There is an integral class $$\widehat{N}= \frac{1}{2} \displaystyle \sum_{i=1}^8 N_i \in \text{Pic}(S).$$ We have a primitive embedding $\Omega_g \hookrightarrow \text{Pic}(Z)$ given by $E \mapsto F, \ \Gamma \mapsto \sigma$ and $\eta \mapsto \widehat{N}$.

\vskip 3pt

If $T_{\Omega_g}$ is the moduli space of $\Omega_g$-polarized $K3$ surfaces considered in \cite{dolgachev}, then it follows by degeneration to $S$ that there is a non-empty open subset $U \seq T_{\Omega_g}$, such that for all $[Y_g] \in U$, we have $\text{Pic}(Y_g) \cong \Omega_g$ and
we may choose the basis $\{ E, \Gamma, \eta \}$ such that $E$, $\Gamma$ are represented by integral effective divisors and furthermore all fibres of $|E|$ are nodal and $\Gamma$ avoids the nodes. Further, if $[Y_g] \in U$ and $\widehat{Y}_g$ denotes the \emph{complex conjugate} (that is, the complex surface obtain by changing the complex structure by sign), then clearly the same statement holds for $\widehat{Y}_g$. The moduli space $T_{\Omega_g}$ has at most two irreducible components, which, locally on the period domain, are interchanged by complex conjugation. Thus, if $Y_g$ is general,  we may choose the basis $\{ E, \Gamma, \eta \}$ such that $E$, $\Gamma$ are represented by integral effective divisors and furthermore all fibres of $|E|$ are nodal and $\Gamma$ avoids the nodes.

\vskip 4pt

It remains to show that we can also choose all fibres of $|E|$ to be irreducible. Suppose $F \in |E|$ is a reducible fibre. Then there is a smooth rational component $R \seq F$. Write $R=aE+b\Gamma+c\eta$ for integers $a,b,c$. We have $(R \cdot E)=0$ as $E$ is nef and $(E)^2=0$, which implies $b=0$ and then $(R)^2=-4c^2$, contradicting $(R)^2=-2$.
\end{proof}

\vskip 4pt

A general surface $Y_g$ as above is endowed with an elliptic fibration $Y_g \to \PP^1$ induced by $|E|$. Recall that the \emph{Mordell--Weil group} $MW(Y_g)$ is the abelian group consisting of the sections of  $Y_g \to \PP^1$. Addition is defined by translation in the fibres, with points on $\Gamma$ serving as origins. The section $\Gamma\in MW(Y_g)$ is the zero element of the Mordell--Weil group.
\vskip 3pt

\begin{lem}
The group $MW(Y_g)$ is infinite cyclic, with generator given by the section
$$ T_1:=2E+\Gamma+\eta. $$ For any  $m \in \mathbb{Z}$, the element $T_m:=mT_1 \in MW(Y_g)$ corresponds to the section $2m^2E+\Gamma+m\eta$.
\end{lem}
\begin{proof}
This will follow immediately from \cite[Theorem  6.3]{SS}, once we show that $$T_m:=2m^2E+\Gamma+m\eta$$ is the class of a irreducible smooth rational curve for any nonzero integer $m$. We have $(T_m)^2=-2$ and $(T_m \cdot E)=1$, so $T_m$ is effective. Assume $T_m$ is not integral. Then there exists a component $R$ of $T_m$, with $(R \cdot T_m) <0$ and $(R)^2=-2$. Write $R=aE+b\Gamma+c\eta$ for integers $a,b,c$. As $E$ is nef, $(R \cdot E) \geq 0$ so $b \geq 0$. Further, $T_m-R$ is effective so $E \cdot (T_m-R) \geq 0$, which gives $b \in \{0, 1 \}$. If $b=0$, then $(R)^2=-4c^2$, so we must have $b=1$. Then $(R)^2=-2$ gives $a=2c^2$ and $(R \cdot T_m)=2(m-c)^2 \geq 0$, which is a contradiction.
\end{proof}

The following result is central to our analysis.

\begin{prop} \label{genus-one}
Let $Y_g$ be general $\Omega_g$-polarized $K3$ surface and $F \in |E|$ be a divisor. The restriction $\eta_{F}$ satisfies $\eta_F^{\otimes \ell} \cong \OO_F$ if and only if $F$ passes through $T_{\ell} \cap \Gamma$. For each $\ell \geq 2$, there are at most finitely many curves $F \in |E|$ such that $\eta_{F}$ is torsion of order $\ell$. If, in addition, the sections $T_m$ and $\Gamma$ meet transversally for all $m=1,\ldots, \ell$, then such curves $F$ exist.
\end{prop}
\begin{proof}
Every fibre $F$ of the fibration $Y_g\rightarrow \PP^1$ is at worst nodal with $\Gamma$ avoiding the nodes and $E_{F}\cong \OO_F$. Thus $\eta_{F}$ satisfies the condition $\eta^{\otimes \ell}_{F} \cong \mathcal{O}_F$ if and only if \ $T_{\ell| F}=\Gamma_{| F}$. As the line bundle $\OO_F(\Gamma_F)\in \mbox{Pic}^1(F)$ has a single section, this occurs if and only if $F$ passes through $T_{\ell} \cap \Gamma$. Since $(T_{\ell} \cdot \Gamma)=2\ell^2-2>0$ for $\ell \geq 2$, there exist finitely many curves $F \in |E|$ such that $\eta_F^{\otimes \ell}\cong \OO_F$. Since $T_1 \cap \Gamma = \emptyset$, the restriction $\eta_{F}$ is never trivial.

Assume now that $T_m$ and $\Gamma$ meet transversally for each $1 \leq m \leq \ell$. There exist $2m^2-2$ elliptic fibres $F$ such that $\eta_{F}$ is nontrivial and $\eta_{F}^{\otimes m} \cong \mathcal{O}_F$. We wish to prove that there exists a curve $F$ such that $\eta_{F}$ is a torsion line bundle of order precisely $\ell$. To that end, for $k\geq 0$, we introduce the arithmetic function \cite[16.7]{HW}
$$\sigma_k(n):= \sum_{d | n} d^k .$$
It suffices to prove that for $n\geq 2$
$$- (2n^2-2)+ \sum_{d | n} (2d^2-2)   < 2n^2-2 .$$
This is equivalent to
$\sigma_2(n) < 2n^2+\sigma_0(n)-2$. Since $\sigma_0(n) \geq 2$ for $n \geq 2$, it suffices to show $\sigma_2(n) < 2n^2$. Let $n= \prod_{i=1}^r p_i^{a_i} $
be the decomposition into distinct prime factors. We have
$$ \sigma_2(n)=\prod_{i=1}^r \frac{p_i^{2(a_i+1)}-1}{p_i^2-1},$$ and then the desired inequality follows from the obvious inequality
$$\frac{p_i^{2(a_i+1)}-1}{p_i^2-1} < 2p_i^{2a_i}. $$
\end{proof}

\begin{definition}\label{singsec}
Let $\pi : S \to \PP^1$ be an elliptic fibration with a section $s:\PP^1 \to S$ and set $\Gamma:=\mathrm{Im}(s)$.  We denote by $\Gamma_{\mathrm{sing}} \seq \Gamma $ the locus of points $t\in \PP^1$ such that $\pi^{-1}(t)$ is singular.
\end{definition}

For the next result, we recall that the M\"obius function $\mu(n)$ is defined by $\mu(p_1\cdots p_k)=(-1)^k$ if $p_1, \ldots, p_k$ are
mutually different primes and $\mu(n)=0$ if $n$ is not square free, see \cite[16.3]{HW}. We also fix integers $g\geq 3$ and $\ell \geq 2$ such that $\ell\geq \sqrt{\frac{g+2}{2}}$. Recall that a node $p$ of a connected nodal curve $C$ is said to be \emph{disconnecting} if the partial normalization of $C$ at $p$ is disconnected.

\begin{prop} \label{nonreduced-prop}
Let $Y_g$ be as in Lemma  \ref{elliptic-1}, set $L=gE+\Gamma$ and consider the Mukai vector
$v=(0,L, 1-g)$. Assume the sections $T_m$ and $\Gamma$ of $Y_g$ meet transversally for $ m=1, \ldots,\ell$. Then there is a $0$-dimensional closed subscheme $Z=Z(Y_g,L,\eta^{\otimes \ell}) \seq M_{Y_g}^s(v)$, nonempty if and only if $g \leq 2\ell^2-2$, classifying curves $C \in |L|$ with $i_*(\eta_C^{\otimes \ell}) \cong i_*(\mathcal{O}_{C})$. All curves $C\in Z$ are nodal and $\eta_{C}\neq \OO_C$.  Furthermore, the underlying set of $Z$ consists of ${ 2\ell^2-2 \choose g }$ points. The number of curves $C\in Z$ such that $\mathrm{ord}(\eta_{C})=\ell$ equals
\begin{equation} \label{nested-sum}
\sum_{d|\ell} \mu\Bigl(\frac{\ell}{d}\Bigr) {2d^2-2 \choose g}>0.
\end{equation}

If we furthermore assume that $\Gamma \cap T_{\ell} \seq \Gamma \setminus \Gamma_{\mathrm{sing}}$, then all nodes of all curves $C \in Z$ are disconnecting.
\end{prop}
\begin{proof}
The line bundle $L\in \mbox{Pic}(Y_g)$ is ample for $g \geq 3$. Furthermore, any $C \in |L|$ has the form $C=\Gamma+E_1+\ldots+E_g$, for $E_i \in |E|$. For such a curve $C$, we denote by $i:C\hookrightarrow Y_g$ the inclusion.  If $\eta^{\otimes \ell}_{C}$ is trivial for $\ell \geq 1$, then certainly $\eta^{\otimes \ell}_{E_i}$ is trivial.

It follows from Proposition \ref{genus-one} that there are at most finitely many divisors $C \in |L|$ with $i_*(\eta_C^{\otimes \ell}) \cong i_*(\mathcal{O}_{C})$. For all such $C$, we have $\eta_{C}\neq \OO_C$. Furthermore, it is easily seen that the conditions of Lemma \ref{locus-via-sheaves} are satisfied for the $K3$ surface $Y_g$, the polarization $H:=L$ and the line bundle $A:=\eta^{\otimes \ell}$. To establish the first statement, we only need to check that $Z\neq \emptyset$ if and only if $g \leq 2\ell^2-2$.

\vskip 6pt

Let $C \in |L|$ be a \emph{reduced} divisor of the form $$C:=\Gamma+ \sum_{i=1}^g E_i,$$ with $E_i \in |E|$ such that $\eta_{E_i}$ is $\ell$-torsion; such divisors exist due to the assumption that $T_{\ell}$ and $\Gamma$ meet transversally and correspond to choosing curves $E_i\in |E|$ passing through one of the $2\ell^2-2$ points of intersections of $T_{\ell}$ and $\Gamma$. There are ${2\ell^2-2 \choose g}$ choices for $C$. The situation can be summarized pictorially in Figure \ref{fig:cover}.

\begin{figure}[h]
\centering
  \includegraphics[width=2.3in]{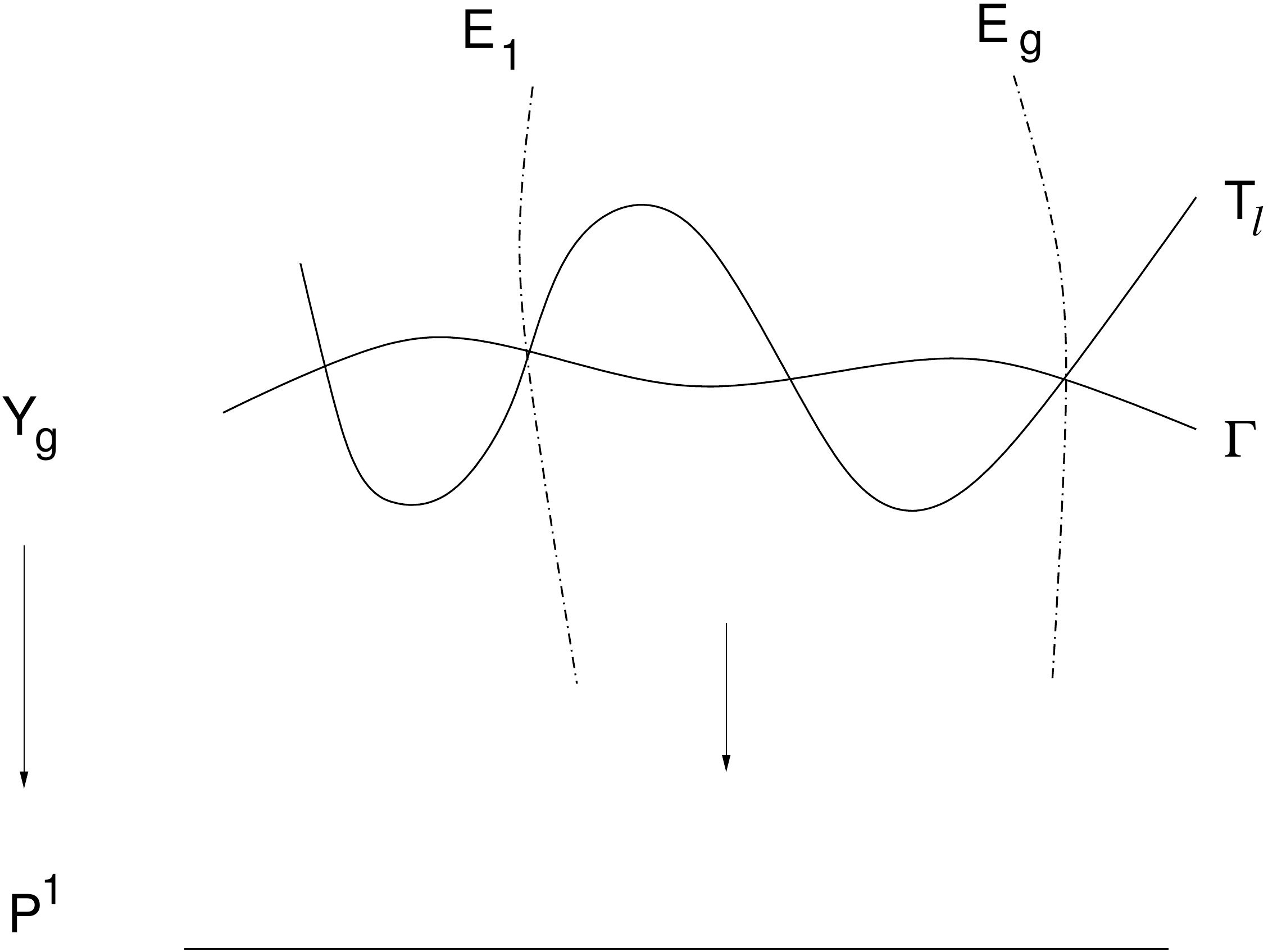}
  \caption{The elliptic surface $Y_g$}
  \label{fig:cover}
\end{figure}

\vskip 3pt

The curve $C$ being \emph{tree-like}, that is, its dual graph becomes a tree after removing self-edges, a line bundle on $C$ is determined by its restrictions to the irreducible components $\Gamma$, $E_1$, \ldots, $E_g$. We have that $\eta^{\otimes \ell}$ restricts to a trivial bundle on each component of $C$, hence $i_*(\eta^{\otimes \ell}) \cong i_*(\mathcal{O}_{C})$. Note also that each such $C$ is nodal. Since $\Gamma \cap T_{\ell} \seq \Gamma \setminus \Gamma_{\text{sing}}$,  each elliptic component of $C$ is smooth, so that each node of $C$ is disconnecting.

\vskip 4pt

We have thus shown that $Z$ is finite and contains at least ${2\ell^2-2 \choose g } $ points. The results of \cite{BV} show that the underlying set of $Z$, when finite, contains at most ${2\ell^2-2 \choose g}$ points, provided $g \leq 2\ell^2-2$. Note also that the assumption (2) on \cite[p.5]{BV}, namely  $H^0(Y_g, L\otimes \eta^{\otimes (-\ell)})=0$, needed in order to carry out the degree calculations in \emph{loc.cit.} holds. Indeed, otherwise there exists $M_1 \in |L-\ell \eta|$. Since $g \leq 2\ell^2-2$, one has $\dim |(2l^2-g)E + \Gamma| \geq 2$ so pick $M_2 \in  |(2\ell^2-g)E + \Gamma|$. Then $T_{-\ell}=M_1+M_2.$ As $T_{-\ell}$ is an integral $(-2)$ curve, this is not possible. We conclude that $Z$ consists of ${2\ell^2-2 \choose g } $ points, provided $g \leq 2\ell^2-2$.

\vskip 4pt
In the case, $g > 2\ell^2-2$, this analysis shows that $Z = \emptyset$. In fact, if $C \in Z$, then $C$ must be nonreduced. Furthermore, we have $C \in Z(Y_g, L , \eta^{\otimes m\ell})$ for any positive integer $m$. But, by the preceding paragraphs, all points in $Z(Y_g, L , \eta^{\otimes m\ell})$ correspond to reduced curves, for $m$ sufficiently large.
\vskip 4pt

For any nontrivial divisor $d|\ell$, there are precisely  ${2d^2-2 \choose g }$  curves $C\in Z$ satisfying  $\eta_C^{\otimes d} \cong \mathcal{O}_{C}$. Let $F(d)$ denote the number of elements $C\in Z$ such that $\mbox{ord}(\eta_C)=d$. We obtain the relation:
$${2\ell^2-2\choose g}=\sum_{d|\ell} F(d).$$
Equation \ref{nested-sum} is then a consequence of the M\"obius inversion formula \cite[16.4]{HW} applied to the arithmetic function $F$. To see that $F(\ell)>0$ for $2\ell^2-2 \geq g$, note that by Proposition  \ref{genus-one}, there exist reduced $C \in |L|$ of the form $C:=\Gamma+ \sum_{i=1}^g E_i$ with $\text{ord}(\eta_{E_1})=\ell$ and such that $\eta_{E_i}$ is $\ell$-torsion for all $i$. For such a $C$, one must have $\text{ord}(\eta_{C})=\ell$.
\end{proof}

\section{Transversality of sections of elliptic surfaces} \label{transversality}

We now verify that the transversality assumption of Proposition \ref{nonreduced-prop} is verified for a general $\Omega_g$-polarized $K3$ surface $Y_g$. This shows that the results of the previous section hold unconditionally, confirms the expectation of \cite{BV} and finishes the proof of Theorem  \ref{generic-transversal}.

\vskip 3pt

Our method of proof is to specialize $Y_g$ to a Kummer surface of the form $\text{Kum}(E \times E)$, for an elliptic curve $E$ with an origin $0\in E$. We denote by $\iota:E\rightarrow E$ the involution given by $\iota(p):=-p$. For $n\in \mathbb Z_{>0}$, we denote by $E[n]$ the set of $n$-torsion points of $E$. Consider the elliptic fibration
$$\text{Kum}(E \times E) \to \PP^1 \cong E / { \iota },$$
given by the projection onto the first factor; we think of $\text{Kum}(E \times E)$ as an elliptic surface with this fixed fibration.  We abuse notation and let $E \seq \text{Kum}(E \times E)$ denote the class of a smooth fibre. The elliptic fibration has a zero section $\Gamma \seq \text{Kum}(E \times E)$, given by the strict transform of the image of the curve
$$\widetilde{\Gamma}:=\bigl\{ (p, 0) \; | \; p \in E\bigr\} \seq E \times E$$ under the quotient map. For each integer $m\neq 0$, we denote by  $T_m \seq \text{Kum}(E \times E)$  the strict transform of the image of the curve $C_m:=\bigl\{ (p, 2m\cdot p) \; | \; p \in E\bigr\} \seq E \times E$.

\begin{lem}
Let $E$ be an elliptic curve and consider the Kummer surface $\mathrm{Kum}(E \times E)$. Then $T_m$ is a smooth section of $\mathrm{Kum}(E \times E)$, and the curves $\Gamma$ and $T_m$ meet transversally in $2m^2-2$ points. Furthermore, $T_m \cap \Gamma \seq \Gamma \setminus \Gamma_{\mathrm{sing}}$. We have the linear equivalence in $\mathrm{Pic}(\mathrm{Kum}(E\times E))$:
$$T_m \equiv 2m(m-1)E+(1-m)\Gamma+mT_1.$$
\end{lem}
\begin{proof}
It follows from \cite[Proposition 3.1]{Sh} that $T_m$ is a smooth section. Further, $C_m$ and $\widetilde{\Gamma}$ meet transversally in $4m^2$ points of the form $(p,0)$, with $p\in E[2m]$. Of these points, precisely $4$ are of the form $(u, 0)$ with $u\in E[2]$. They do not contribute to the intersection of $T_m$ and $\Gamma$. Hence $T_m$ and $\Gamma$ meet transversally at $2m^2-2$ points. None of these points lie in $\Gamma_{\mathrm{sing}}$ (recall Definition \ref{singsec}).


\vskip 4pt

To compute the class of $T_m$, let $\Lambda \seq \text{Pic}(\text{Kum}(E \times E))$ denote  the lattice generated by $E$, $\Gamma$ and all fibre components avoiding $\Gamma$. From \cite[Proposition 3.1]{Sh} and  \cite[Theorem 6.3]{SS}, we have the expression
$T_m=mT_1+y$, for some $y \in \Lambda$. Write $y=aE+b\Gamma+D$ for integers $a,b$, where $D$ is a sum of fibre components avoiding $\Gamma$. As $T_m$ and $T_1$ are sections, intersecting with $E$ shows that  $b=1-m$. Intersecting with $\Gamma$ then gives $2m^2-2=a-2(1-m)$, so $a=2m(m-1)$. Lastly, one easily computes $(T_1 \cdot T_m)=2(m-1)^2-2$ in the  same manner as the computation of $(\Gamma \cdot T_m)$. After evaluating $(T_m-mT_1)^2$, this produces $(D)^2=0$, which implies that $D$ is trivial, since $D$ is a sum of fibre components avoiding $\Gamma$.
\end{proof}

To finish the proof, we only need to show that the embedding $$\phi \; : \; \Omega_g \hookrightarrow \mathrm{Pic}(\mathrm{Kum}(E \times E))$$ given by $E \mapsto E$, $\Gamma \mapsto \Gamma$ and $\eta \mapsto T_1-2E-\Gamma$ is primitive (since $T_1$ does not lie in $\Lambda$ by \cite[Proposition 3.1]{Sh}, the map $\phi$ is indeed an embedding). It will then follow by specialization to $\mbox{Kum}(E \times E)$ that the transversality assumption of Proposition \ref{nonreduced-prop} is verified for a general surface $Y_g$.

\begin{lem} \label{embedding-is-prim}
The embedding $\phi \; : \; \Omega_g \hookrightarrow \mathrm{Pic}(\mathrm{Kum}(E \times E))$ is primitive.
\end{lem}
\begin{proof}
Let $\Delta$ denote the smallest primitive sublattice of $\text{Pic}(\mathrm{Kum}(E \times E))$ containing the image of $\Omega_g$, that is,
$$ \Delta= \text{Im}(\Omega_g) \otimes_{\mathbb{Z}} \mathbb{Q} \ \cap \ \text{Pic}(\mathrm{Kum}(E \times E)).$$
Since $\Omega_g$ has discriminant $4$, it follows that $\text{disc}(\Delta)=\frac{4}{k^2}$ for some positive integer $k$. If $\phi$ is not primitive, we must have $\text{disc}(\Delta)=1$, which implies that $\Delta$ is a rank $3$, even, unimodular lattice. But no such lattice exists since the signature of any even, unimodular lattice is necessarily divisible by $8$, see \cite[Chapter II, Theorem 5.1]{HM}.
\end{proof}

\vskip 5pt
We will also need the following facts:

\begin{lem} \label{non-eff}
Let $X_g$ be a Barth--Verra surface of genus $g$. Then the class $c \eta \in \mathrm{Pic}(X_g)$ is not effective, for any $c\in \mathbb Z$. Furthermore, every divisor $D \in |L+c\eta|$ is integral.
\end{lem}
\begin{proof}
Indeed, suppose $c \eta$ is effective. Since $(c\eta)^2=-4c^2$, there exists an integral component $R$ of $c \eta$ with $(R \cdot c \eta)<0$, so that $R$ is not nef and thus $(R)^2=-2$. But then $(L \cdot \eta)=0$ and $L$ is nef, so we must have $(R \cdot L)=0$. Thus if $R=aL+b\eta$, for integers $a,b$, one must have $a=0$ so $(R)^2=-4b^2$, contradicting that $(R)^2=-2$.

\vskip 3pt

Now assume $|L+c\eta|\neq \emptyset $, and suppose we have $D_1+D_2 \in |L+c\eta|$, for effective divisors $D_i=a_iL+b_i \eta$, $i=1,2$ and integers $a_i, b_i$. Intersecting with the nef line bundle $L$, it follows $a_i \geq 0$. Since $a_1+a_2=1$, we may assume $a_1=0$. As $D_1=b_1 \eta$ is effective, we must have $b_1=0$, that is, $D_1=0$. Thus all divisors $D\in |L+c\eta|$ are integral.
\end{proof}

\vskip 5pt
In the interest of completeness, we record the following elementary lemma.
\begin{lem} \label{discon}
Let $\pi \; : \mathcal{C} \to B$ be a flat, proper family of connected, nodal curves of genus $g$. Let $0 \in B$, and assume all the nodes of $\pi^{-1}(0)$ are disconnecting and that $\pi^{-1}(t)$ is integral, for $t \in B-\{0\}$. Then there is an open subset $U \seq B$ such that $\pi^{-1}(t)$ is smooth, for $t  \in U-\{0\}$.
\end{lem}
\begin{proof}
Denote by $m:B\rightarrow \mm_g$ the moduli map induced by $\pi$. The hypothesis implies that the point $m(0)$ corresponding to the stable model of $\pi^{-1}(0)$ does not lie in the boundary divisor $\Delta_0$ of $\mm_g$ of irreducible singular curves and their degenerations. Then there exists an open set $0\in U\subset B$, such that $m(t):=[\pi^{-1}(t)]\in \mm_g-\Delta_0$, for all $t\in U$. Since $\pi^{-1}(t)$ is assumed to be integral, the conclusion follows.

\end{proof}

\vskip 5pt

\begin{proof}[Proof of Theorem \ref{generic-transversal}]
As already explained in the introduction, we can deform a general Barth--Verra surface $X_g$ to a general $\Omega_g$-polarized surface $Y_g$. By Lemma \ref{non-eff}, all divisors $C\subset X_g$ in the linear system $|L|$ are integral. The result follows from Lemma \ref{locus-via-sheaves}, Proposition \ref{nonreduced-prop} and the transversality result above. Indeed, from Proposition \ref{nonreduced-prop} all curves from $Z(Y_g,L, \eta^{\otimes \ell})$ are nodal with only disconnecting nodes. This ensures that Lemma \ref{discon} can be applied.  From \cite{BV} we have that the support of $Z=Z(X_g,L,\eta^{\otimes \ell})$ contains at most ${2\ell^2-2 \choose g}$ points, whereas by deformation to  $Y_g$ and Proposition \ref{nonreduced-prop}, $Z$ contains at least ${2\ell^2-2 \choose g}$ points. By Lemma \ref{discon} all curves corresponding to points in $Z$ are smooth.
\end{proof}

\section{Syzygies of paracanonical curves on Barth--Verra $K3$ surfaces}

In this section we prove Theorems \ref{pg} and \ref{pgeven}. We fix integers $g\geq 7$ and $\ell\geq \sqrt{\frac{g+2}{2}} $ and consider a  Barth--Verra $K3$ surface $X_g$ of genus $g$ having $\mbox{Pic}(X_g)=\Upsilon_g=\mathbb Z\cdot L\oplus \mathbb Z\cdot \eta$ as in Definition \ref{bv2}. We set $H:=L+\eta\in \mbox{Pic}(X_g)$, thus the smooth curves from the linear system $|H|$ have genus $g-2$.

\vskip 4pt



\vskip 4pt

Using Theorem \ref{generic-transversal}, there exists a smooth  genus $g$ curve $C\in |L|$,  such that $\mbox{ord}(\eta_C)=\ell$. In particular, $[C, \eta_C]\in \cR_{g,\ell}$ and
$$\phi_{|H_C|}:C\hookrightarrow \PP^{g-2}$$
is a level $\ell$ paracanonical curve. We shall verify Theorems \ref{pg} and \ref{pgeven} for $[C, \eta_C]$, depending on whether $g$ is odd or even. Our proof follows along the lines of \cite{FK} and here we just outline the main steps, while highlighting the differences.

\vskip 4pt
We first need a lemma.
\begin{lem} \label{needed-for-green}
Assume $g \geq 7$. Then $H^1(X_g, qH-L)=0$ for $q \geq 1$.
\end{lem}
\begin{proof}
For $q=1$, we write $H-L=\eta$, which has no first cohomology from Lemma \ref{non-eff}. We now prove the claim by induction on $q \geq 2$. We have $2H-L=L+2\eta$. Suppose for a contradiction that $h^1(X_g, L+2\eta)>0$. As we are assuming $g \geq 7$, we have $(L+2\eta)^2 \geq -4$ and furthermore $L \cdot (L+2\eta)>0$, so $L+2\eta$ is effective by Riemann--Roch. Choose any $D \in |L+2\eta|$. We obtain $h^0(\mathcal{O}_D)=h^1(X_g, \mathcal{O}_{X_g}(-D))+1=h^1(L+2\eta)+1>1$, by assumption. This contradicts Lemma \ref{non-eff}.

We now prove the induction step. Choose a smooth element $D \in |H|$ is smooth. For $q \geq 2$, we have a short exact sequence
$$ 0 \longrightarrow \mathcal{O}_{X_g}(qH-L) \longrightarrow \mathcal{O}_{X_g}((q+1)H-L) \longrightarrow  \omega_D\bigl((q-1)H+\eta\bigr) \longrightarrow 0.$$ Since $H \cdot ((q-1)H+\eta)>0$, the claim follows by induction.

\vskip 3pt





\end{proof}

\vskip 5pt

Using the vanishing provided by Lemma \ref{needed-for-green}, we are in a position to write down Green's exact sequence \cite[Theorem 3.b.1]{G} of Koszul cohomology groups on $X=X_g$:

\begin{equation}\label{grseq}
\cdots \longrightarrow K_{p,q}(X,H)\longrightarrow K_{p,q}(C,H_C)\longrightarrow K_{p-1,q+1}(X,-C,H)\longrightarrow \cdots,
\end{equation}
where the group $K_{p-1,q+1}(X,-C,H)$ is computed by the following part of the Koszul complex:
$$\cdots \longrightarrow \bigwedge^p H^0(X,H)\otimes H^0(X,qH-C)\stackrel{d_{p,q}}\longrightarrow \bigwedge^{p-1} H^0(X,H)\otimes H^0(X, (q+1)H-C)\stackrel{d_{p-1,q+1}}\longrightarrow
$$
$$\stackrel{d_{p-1,q+1}}\longrightarrow \bigwedge^{p-2} H^0\bigl(X,H)\otimes H^0(X, (q+2)H-C)\longrightarrow \cdots.$$

\vskip 4pt

\subsection{The Prym--Green Conjecture in genus $g=2i+5$.}

According to \cite{CEFS}, the naturality of the resolution of the paracanonical curve $C\subset \PP^{g-2}$ is equivalent to the following statements:
$$K_{i+1,1}(C,H_C)=0 \ \ \mbox{ and } \ \ K_{i-1,2}(C,H_C)=0.$$ Using the sequence (\ref{grseq}) for $(p,q)=(i+1,1)$ and $(p,q)=(i-1,2)$ respectively, it suffices to show
\begin{equation}\label{van1}
K_{i+1,1}(X,H)=0 \  \mbox{ and } \ \ K_{i-1,2}(X,H)=0,
\end{equation}
respectively
\begin{equation}\label{van2}
K_{i,2}(X,-C,H)=0 \ \mbox{ and } \ K_{i-2,3}(X,-C,H)=0.
\end{equation}

\begin{prop}\label{odd1}
For a Barth--Verra surface $X_g$ of genus $g=2i+5$, we have that:
$$K_{i+1,1}(X,H)=0 \  \mbox{ and } \ \ K_{i-1,2}(X,H)=0.$$
\end{prop}
\begin{proof} Building upon Voisin's fundamental work \cite{V1} and \cite{V2}, the Koszul cohomology of \emph{every} polarized $K3$ surface $(X,H)$ is completely determined in \cite{AF}. Precisely, if $D\in |H|$ is a general element of the linear system, in accordance with Green's Conjecture, it is shown in \cite{AF} that
$$K_{p,2}(X,H)=0, \ \mbox{ for } p<\mathrm{Cliff}(D), \ \mbox{ and }$$
$$K_{p,1}(X,H)=0, \ \mbox{ for } p\geq 2i+2-\mathrm{Cliff}(D).$$
Since $D$ is a curve of genus $g-2=2i+3$, it suffices to show that the Clifford index of $D$ is maximal, that is, $\mbox{Cliff}(D)=i+1$. But Lemma \ref{non-eff}, together with a well-known result of Lazarsfeld \cite{lazarsfeld-bnp}, implies that the general element of $|H|$ is smooth and Brill--Noether--Petri general. In particular, it has maximal Clifford index. 
\end{proof}

We recall that if $D$ is a smooth curve and $L\in \mbox{Pic}(D)$ is a globally generated line bundle, one denotes by $M_D$ the kernel bundle defined by the exact sequence:
\begin{equation}\label{lazbun}
0\longrightarrow M_D\longrightarrow H^0(D,L)\otimes \OO_D\stackrel{\mathrm{ev}}\longrightarrow L\longrightarrow 0.
\end{equation}

\begin{prop}\label{odd2}
For a general Barth--Verra surface $X_g$ of genus $g=2i+5$, the following holds:
$$K_{i,2}(X,-C,H)=0 \ \mbox{ and } \  \ K_{i-2,3}(X,-C,H)=0.$$
\end{prop}
\begin{proof}
Using \cite[Lemma 2.2]{FK}, we can restrict the above Koszul cohomology groups to a general curve $D\in |H|$ to obtain isomorphisms:
$$K_{i,2}(X,-C,H)\cong K_{i,2}(D, -C_D, K_D) \ \mbox{ and } \ K_{i-2,3}(X, -C,H)\cong K_{i-2,3}(D, -C_D, K_D).$$
By taking exterior powers and then cohomology in the exact sequence (\ref{lazbun}), we obtain:
$$K_{i,2}(D, -C_D,K_D)\cong H^0\Bigl(D, \bigwedge ^i M_{K_D}\otimes (2K_D-C_D)\Bigr) \ \ \ \ \mbox{ and } $$
$$K_{i-2,3}(D,-C_D, K_D)\cong H^1\Bigl(D, \bigwedge^{i-1} M_{K_D}\otimes (2K_D-C_D)\Bigr)\cong H^0\Bigl(D, \bigwedge ^{i-1} M_{K_D}^{\vee}\otimes (C_D-K_D)\Bigr)^{\vee}.$$
By direct calculation we obtain the following slopes:
$$\mu\Bigl(\bigwedge^{i-1} M_{K_D}^{\vee}\otimes (C_D-K_D)\Bigr)=2i+2=g(D)-1, \ \ \mbox{ and }$$
$$\mu\Bigl(\bigwedge ^i M_{K_D}\otimes (2K_D-C_D)\Bigr)=2i<g(D)-1.$$
In particular, both these vector bundles are expected to have no non-trivial global sections. To check this fact for one particular Barth--Verra surface, one proceeds exactly like in
\S 3 of \cite{FK}\footnote{We take this opportunity to point out a typo in \cite[Corollary 3.4]{FK}: the conclusion should be $K_{j,2}(Z_g,H)=0$ for $j \leq p$.}. Namely, we specialize $X$ to a \emph{hyperelliptic} $K3$ surface $\widehat{X}$ with Picard lattice $\widehat{\Upsilon}_g$ having ordered basis $\{L, \eta, E\}$ and intersection form:
\[ \left( \begin{array}{ccc}
4i+8 & 0 & 2 \\
0 & -4  &0 \\
2 & 0 & 0
\end{array} \right).\]
Since $\Upsilon_g$ can be primitively embedded in $\widehat{\Upsilon}_g$, the surface $X$ can be specialized to $\widehat{X}$, where both curves $C$ and $D$ become hyperelliptic. In particular, $M_{K_D}$ splits as  $\bigl(\OO_D(E_D)^{\vee}\bigr)^{\oplus (2i+2)}$, where $E_D$ is the degree $2$ pencil on $D$, and one checks directly like in \cite[Lemma 3.5]{FK} that
$$ H^0\Bigl(D, \bigwedge ^{i-1} M_{K_D}^{\vee}\otimes (C_D-K_D)\Bigr)=0 \ \mbox{ and } \ H^0\Bigl(D, \bigwedge ^i M_{K_D}\otimes (2K_D-C_D)\Bigr)=0.$$
\end{proof}

\noindent \emph{Proof of Theorem \ref{pg}.} Propositions \ref{odd1} and \ref{odd2} complete the proof of Theorem \ref{pg} for $g\geq 7$. For $g=5$, the Prym--Green Conjecture amounts to the one single statement
$$K_{1,1}(C,K_C\otimes \eta)=0,$$
or equivalently, a general level paracanonical level $\ell$ curve $C\hookrightarrow \PP^3$ of genus $5$ lies on no quadric surfaces. Since such a quadric has rank at most $4$, this is equivalent to the statement that for a general $[C, \eta]\in \cR_{5,\ell}$, the torsion point $\eta$ cannot be written as the difference of two pencils from $W^1_4(C)$ . This is a simple exercise that can be solved via limit linear series, by degenerating $C$ to a curve of compact type $C'\cup E$, consisting of an  elliptic curve $E$ and a genus $4$ curve $C'$ .
\hfill $\Box$

\vskip 6pt

\subsection{The Prym--Green Conjecture in genus $g=2i+6$.}  Using a level $\ell$ curve  $[C, \eta_C]\in \cR_{g,\ell}$ lying either on a general Barth--Verra surface $X$ as above (when $\ell\geq \sqrt{\frac{g+2}{2}}$), or on a Nikulin surface when $\ell=2$ as in \cite{FK},  we are able to show that
$$K_{i+2,1}(C,K_C\otimes \eta_C)=0 \mbox{ and } \ K_{i-1,2}(C,K_C\otimes \eta)=0.$$ As pointed out in the introduction, this result is not quite optimal, for the Prym-Green Conjecture predicts the stronger vanishing $K_{i,2}(C,K_C\otimes \eta)=K_{i+1,1}(C,K_C\otimes \eta)=0$ (at least when $\ell\geq 3$). This result \emph{cannot} be achieved with the methods of this paper, see Remark \ref{failure}.

\vskip 4pt

\noindent \emph{Proof of Theorem \ref{pgeven}}.
Using (\ref{grseq}) the conclusion follows once we prove that
$$K_{i-1,2}(X,H)=0 \ \mbox{ and } \ K_{i+2,1}(X,H)=0,$$
as well as,
$$K_{i-2,3}(X,-C,H)=0 \ \mbox{ and } \ K_{i+1,2}(X,-C,H)=0.$$
These are the precise analogues in even genus of Propositions \ref{odd1} and \ref{odd2} respectively and the proofs are identical, so we skip the details.
\hfill $\Box$

\begin{remark}\label{failure}
Keeping the same notation as above, using (\ref{grseq}) we can write the exact sequence:
$$\cdots K_{i+1,1}(X,-C,H)\longrightarrow K_{i+1,1}(X,H)\longrightarrow K_{i+1,1}(C,K_C\otimes \eta_C)\longrightarrow \cdots.$$
Clearly $K_{i+1,1}(X,-C,H)=0$ whereas $K_{i+1,1}(X,H)\neq 0$, therefore $K_{i+1,1}(C,K_C\otimes \eta)\neq 0$ as well. In other words paracanonical curves of genus $2i+6$ on Nikulin or Barth--Verra surfaces \emph{do not} satisfy property $(N_i)$ predicted by the Prym--Green Conjecture.
\end{remark}


\begin{thebibliography}{aaaaaa}
\bibitem[AF]{AF} M. Aprodu and G. Farkas, {\emph{The Green Conjecture for smooth curves lying on arbitrary $K3$ surfaces}}, Compositio Math. \textbf{147} (2011), 839-851.
\bibitem[BV]{BV} W. Barth and A. Verra, {\emph{Torsion on $K3$-sections}},  Problems in the theory of surfaces and their classification
(Cortona, 1988), Symposia Mathematica, Vol. 32, Academic Press 1991, 1-24.
\bibitem[CF]{CF} A. Chiodo and G. Farkas, {\em{Singularities of the moduli space of level curves}}, arXiv:1205.0201, to appear in the Journal of the European Math. Society.
\bibitem[CEFS]{CEFS} A. Chiodo, D. Eisenbud, G. Farkas and F.-O. Schreyer, {\em{Syzygies of torsion bundles and the geometry of the level $\ell$ modular variety over $\overline{\mathcal{M}}_g$}}, Inventiones Math. \textbf{194} (2013), 73-118.
\bibitem[Dol]{dolgachev} I. Dolgachev, {\em{Mirror symmetry for lattice polarized $K3$ surfaces}}, Journal of Mathematical Sciences \textbf{81} (1996), 2599-2630.

\bibitem[FK]{FK} G. Farkas and M. Kemeny, {\em{The generic Green-Lazarsfeld Secant Conjecture}}, Inventiones Math. \textbf{203} (2016), 265-301.
\bibitem[FL]{FL} G. Farkas and K. Ludwig, {\em{The Kodaira dimension of the moduli
space of Prym varieties}}, Journal of the European Math. Society \textbf{12} (2010), 755-795.
\bibitem[FV]{FV} G. Farkas and A. Verra, {\em{Moduli of theta-characteristics via Nikulin surfaces}}, Mathematische Annalen \textbf{354} (2012), 465-496.
\bibitem[vGS]{vGS} B. van Geemen and A. Sarti, {\em{Nikulin involutions on $K3$ surfaces}}, Mathematische Zeitschrift \textbf{255} (2007), 731-753.
\bibitem[G]{G} M. Green, {\em{Koszul cohomology and the cohomology of projective varieties}}, Journal of Differential Geometry \textbf{19} (1984), 125-171.
\bibitem[HW]{HW} {G. Hardy and E. Wright}, {\em{An introduction to the theory of numbers}}, Fifth Edition, Oxford Science Publications 1979.
\bibitem[HK]{HK} D. Huybrechts and M. Kemeny, {\em{Stable maps and Chow groups}}, Documenta Math. \textbf{18} (2013), 507-517.
\bibitem[HL]{HL} D. Huybrechts and M. Lehn, {\em{The geometry of the moduli space of sheaves}}, Cambridge University Press 2010.
\bibitem[HM]{HM} J. Milnor and D. Husemoller, {\emph{Symmetric bilinear forms}}, Ergebnisse der Mathematik und ihrer Grenzgebiete 73, Springer-Verlag, New York-Heidelberg, 1973.
\bibitem[La]{lazarsfeld-bnp} R. Lazarsfeld, {\em{Brill-Noether-Petri without degenerations}}, Journal of Differential Geometry \textbf{23} (1986), 299-307.
\bibitem[Mo]{Mo} D. Morrison, {\em{On $K3$ surfaces with large Picard number}}, Inventiones Math. \textbf{75} (1984), 105-121.
\bibitem[SS]{SS} M. Sch\"utt and T. Shioda, {\emph{Elliptic surfaces}}, Algebraic geometry in East Asia (Seoul 2008), Advanced Studies
in Pure Mathematics, Vol. 60, 2010, 51-160.
\bibitem[Sh]{Sh} T. Shioda,  {\emph{Correspondence of elliptic curves and Mordell-Weil lattices of certain elliptic $K3$ surfaces}}, Algebraic
cycles and motives, 2007, 319-339, Cambridge University Press.
\bibitem[SP]{stacks} The {Stacks Project Authors}, {\itshape Stacks Project}, http://stacks.math.columbia.edu, 2016.
\bibitem[V1]{V1} C. Voisin, {\em{Green's generic syzygy conjecture for curves of even genus lying on a $K3$ surface}}, Journal of the
European Math. Society \textbf{4} (2002), 363-404.
\bibitem[V2]{V2} C. Voisin, {\em{Green's canonical syzygy conjecture for generic curves of odd genus}},
Compositio Math. \textbf{141} (2005), 1163--1190.
\end{thebibliography}
\end{document}